\documentclass[reqno]{amsart}

\usepackage{amsmath}
\usepackage{amsfonts}
\usepackage{amssymb,enumerate,comment,mathdots, stmaryrd}
\usepackage{amsthm}
\usepackage[all]{xy}
\usepackage{rotating}
\usepackage{hyperref}

\usepackage{mathrsfs}

\usepackage{graphicx}
\usepackage{tcolorbox}

\theoremstyle{plain}
\newtheorem{lem}{Lemma}[section]

\newtheorem{prop}[lem]{Proposition}
\newtheorem{thm}[lem]{Theorem}

\theoremstyle{definition}
\newtheorem{defn}[lem]{Definition}
\newtheorem{ex}[lem]{Example}
\newtheorem{question}[lem]{Question}

\newtheorem{disc}[lem]{Remark}

\newtheorem{notn}[lem]{Notation}
\newtheorem{fact}[lem]{Fact}

\newtheorem{assumption}[lem]{Assumption}

\theoremstyle{remark}

\newcommand{\Hom}{\operatorname{Hom}}

\newcommand{\s}{\mathfrak{S}}

\newcommand{\Cl}{\operatorname{Cl}}
\newcommand{\Pic}{\operatorname{Pic}}

\newcommand{\ideal}[1]{\mathfrak{#1}}
\newcommand{\m}{\ideal{m}}

\newcommand{\fm}{\ideal{m}}

\newcommand{\fa}{\ideal{a}}
\newcommand{\fb}{\ideal{b}}

\newcommand{\sfk}{\mathsf k}

\newcommand{\ol}{\overline}

\newcommand{\bbz}{\mathbb{Z}}

\newcommand{\vf}{\varphi}

\newcommand{\tri}{\trianglelefteq}

\renewcommand{\geq}{\geqslant}
\renewcommand{\leq}{\leqslant}

\newcommand{\Ext}[4][R]{\operatorname{Ext}_{#1}^{#2}(#3,#4)}

\newcommand{\Otimes}[3][R]{#2\otimes_{#1}#3}
\renewcommand{\Hom}[3][R]{\operatorname{Hom}_{#1}(#2,#3)}

\newcommand{\OAtimes}[3][A]{#2\otimes_{#1}#3}

\numberwithin{equation}{lem}

\newcommand{\ec}[1][C]{\ol{[#1]}}
\newcommand{\sbar}[1][R]{\ol{\s}_0(#1)}

\begin{document}

\bibliographystyle{amsplain}

\author{Sean K. Sather-Wagstaff}

\address{School of Mathematical and Statistical Sciences,
Clemson University,
O-110 Martin Hall, Box 340975, Clemson, S.C. 29634,
USA}

\email{ssather@clemson.edu}

\urladdr{https://ssather.people.clemson.edu/}

\author{Tony Se}

\address{Department of Mathematics,
West Virginia University,
320 Armstrong Hall,
P.O. Box 6310,
Morgantown, WV 26506-6310,
USA}

\email{tony.se@mail.wvu.edu}

\urladdr{http://community.wvu.edu/~tts00001/}

\author{Sandra Spiroff}

\address{Department of Mathematics,
University of Mississippi,
Hume Hall 335, P.O. Box 1848, University, MS 38677,
USA}

\email{spiroff@olemiss.edu}

\urladdr{http://math.olemiss.edu/sandra-spiroff/}

\thanks{
Sandra Spiroff was supported in part by Simons Foundation Collaboration Grant 584932.}

\title{Generic Constructions and Semidualizing Modules}

\date{\today}



\begin{abstract} 
We investigate some
general machinery for describing semidualizing modules over generic constructions like ladder determinantal rings with
coefficients in a normal domain. We also pose and investigate natural localization questions that arise in the process.
\end{abstract}

\maketitle


\section*{Introduction} \label{sec00}

Throughout this paper, 
all rings are commutative  with identity, and $\sfk$ is a field.

\

This paper investigates 
some general machinery for describing semidualizing modules over generic constructions.
Recall that a \emph{semidualizing module} over a noetherian ring $R$ is a finitely generated module $C$ such that
$\Hom CC\cong R$
and $\Ext{\geq 1}CC=0$. Examples include the free $R$-module $R$ and a canonical module $\omega$
over a 
Cohen-Macaulay ring.
Introduced by Foxby~\cite{foxby:gmarm}, these modules allow, among other things, for investigation into 
homological properties of modules that, in the case of $R$ or $\omega$, recover duality results of
Auslander and Bridger~\cite{auslander:smt}
and Grothendieck (as documented by Hartshorne~\cite{hartshorne:lc}).
Other applications are described in~\cite{avramov:rhafgd,sather:cidfc,sather:bnsc}. 
Given the utility of these objects, we work toward the goal of describing all semidualizing modules over certain classes of rings. 

We focus on rings that are constructed generically,
that is, families of rings defined over $\bbz$, then broadened to coefficients in other rings. 
As an example of a generic construction, let $X$ be an $m\times n$ matrix of indeterminates,
let $t$ be a positive integer, and let $\bbz[X]$ be the polynomial ring over  $\bbz$ with indeterminates in $X$.
Then the ideal $I_t(X)$ generated by the 
size $t$ minors of $X$ is defined over $\bbz$, but is often studied over $\sfk$. 
In~\cite{sather:divisor}, the semidualizing modules over the ring $A_t(X)=A[X]/I_t(X)$ are completely determined when $A$ is a field and,
moreover, when $A$ is a  normal domain (i.e., a noetherian integrally closed integral domain) by 
using the divisor class group to combine the semidualizing modules over $A$ with those of $\sfk_t(X)$. 

The point of the current paper is to build a foundation for simplifying 
the constructions in~\cite{sather:divisor}, in a way that applies to the more general
ladder determinantal context.
(See the examples below and the results in~\cite{SWSeSpP4}
for details about these constructions and our computations of their semidualizing modules.)

The  main idea in our simplification is to use the fact that these rings are tensor products over a principal ideal domain $D$.
Accordingly, Sections~\ref{sec181225a} and~\ref{sec191006c} of the current paper describe how semidualizing modules over two $D$-algebras
give rise to semidualizing modules over the tensor product algebra. The first of these sections deals with this general setting, while the second  
assumes that the tensor product is normal. 

One aspect of our work in~\cite{SWSeSpP4} is to localize our tensor product strategically. Since the rings considered there are
standard graded normal domains, certain localization properties are automatically especially nice; see Fact~\ref{prop190312a} below.
This led us to consider similar localizations in non-normal standard graded rings, where we were somewhat surprised to learn that things are not so nice. 
This is the subject of Section~\ref{sec191006b}, which contains a few open questions
and the surprising Example~\ref{ex190329b}.


\section{Background}\label{sec190813a}


\begin{fact}\label{fact181224a}
Let $D$ be a ring, and let $R_1,R_2$ be $D$-algebras.
Set $T=R_1\otimes_DR_2$. 
For $i=1,2$, let $M_i$ be an $R_i$-module.
Then we have the following natural $T$-module isomorphisms, the first and third of which are tensor cancellation and associativity.
\begin{align*}
(M_1\otimes_{R_1}T)\otimes_T(T\otimes_{R_2}M_2)
&\cong M_1\otimes_{R_1} T\otimes_{R_2}M_2\\
&\cong M_1\otimes_{R_1}(R_1\otimes_DR_2)\otimes_{R_2}M_2\\
&\cong M_1\otimes_DM_2
\end{align*}
The second isomorphism is by the definition of $T$. 
\end{fact}

\begin{disc}\label{disc181224a}
If $R$ is an integral domain, then $R$ either contains a field or an isomorphic copy of $\bbz$.
In particular, $R$ contains a subring $D$ that is  a principal ideal domain. (In this paper, we consider fields to be principal ideal domains.)
With this set-up, every torsion-free $R$-module is torsion-free over $D$, hence flat over $D$; in particular, $R$ is flat over $D$, as is every reflexive $R$-module.
\end{disc}

\begin{notn}\label{notn191006a}
Given a noetherian ring $R$, we let $\s_0(R)$ denote the set of isomorphism classes of semidualizing $R$-modules.
\end{notn}

\begin{disc}\label{disc191006b}
If $R$ is noetherian and local, then $\s_0(R)$ is finite by~\cite[Theorem~A]{nasseh:gart}.
The finiteness can fail if $R$ is not local; see Example~\ref{ex190329a}. 
See, however, Fact~\ref{prop190312a} for an important non-local case where finiteness does hold.
A version of $\s_0(R)$ that is suspected to always be finite is described below; see Remark~\ref{disc191101b}\eqref{disc191101b2}.
\end{disc}

The set $\s_0(R)$ has more structure, as we describe next.
A feature of our results in~\cite{sather:divisor,SWSeSpP4,SWSeSpP1,SWSeSpP2} is that 
we can also keep track of this additional structure.

\begin{defn}\label{defn191006a}
Let $C$ and $M$ be  finitely generated
modules over a noetherian ring $R$. Then $M$ is 
\emph{totally $C$-reflexive} if $M \cong \Hom{\Hom MC}C$ 
and $\Ext{\geq 1}MC=0=\Ext{\geq 1}{\Hom MC}C$.
If $M$ and $C$ are 
semidualizing, then we relate the isomorphism classes $[M],[C]\in\s_0(R)$ as $[C]\tri[M]$ provided that
$M$ is totally $C$-reflexive.
\end{defn}

\begin{disc}
When $R$ is noetherian and local, the relation $\tri$ on $\s_0(R)$ is reflexive and antisymmetric by~\cite[(5.3)]{takahashi:hiatsb}. 
Antisymmetry can fail when $R$ is not local; see, however, Remark~\ref{disc191101b}\eqref{disc191101b3}.
It is suspected
that $\tri$ is transitive; see \cite[Question~2.5]{sather:sdm}.
\end{disc}

For noetherian rings $R$ and $S$, we use Definition~\ref{defn191006a} with  the following to build a relation on  $\s_0(R)\times\s_0(S)$.

\begin{notn}\label{disc181225a}
Let $U$ and $V$ be sets with relations denoted $\tri$. Then we consider the product relation on $U\times V$,
where $(u,v)\tri(u',v')$ provided that $u\tri u'$ and $v\tri v'$. We write $U\approx V$ if there is 
a \emph{perfectly relation-respecting bijection} from $U$ to $V$, i.e., a bijection $f\colon U\to V$ such that $u\tri u'$ if and only if 
$f(u)\tri f(u')$ for all $u,u'\in U$.
\end{notn}

\begin{disc} \label{disc:flat1}
Let $\vf\colon R\to S$ be a flat ring homomorphism between noetherian rings. 
It is straightforward to show that the map
$\s_0(\vf)\colon\s_0(R)\to\s_0(S)$ given by $[C]\mapsto[\Otimes CS]$ is well-defined and \emph{relation-respecting},
i.e., $[C]\tri[M]$ implies $[\Otimes CS]\tri[\Otimes MS]$ for all $[C],[M]\in\s_0(R)$.
\end{disc}

We continue this section by discussing how semidualizing modules and divisor class groups interact.

\begin{defn}
The \emph{divisor class group} $\Cl(A)$ of a  normal domain $A$ is the set of isomorphism classes
of finitely generated rank-1 reflexive $A$-modules with operation
$[\fa]+[\fb]=[(\fa\otimes_A\fb)^{**}]$.  The additive identity is $[A]$, the inverse of $[\fb]$ is $-[\fb]=[\Hom[A]\fb A]$, and 
$[\fa]-[\fb]=[\Hom[A]\fb\fa]$.
\end{defn}

\begin{disc} \label{disc:flat2}
It is a well-known fact that if $\vf\colon A\to B$ is a flat ring homomorphism between normal domains, then the map
$\Cl(\vf)\colon\Cl(A)\to\Cl(B)$ given by $[\fa]\mapsto[\OAtimes \fa B]$ is a well-defined group homomorphism.  The lemma below relates $\s_0(A)$ and $\Cl(A)$.
\end{disc}

\begin{lem}\label{lem190103a}
Let $A$ be a  normal domain, and let $N$ be a finitely generated $A$-module. 
Then $N$ is semidualizing over $A$ if and only if it is $A$-reflexive of rank 1 and $\Ext[A]{\geq 1}NN=0$.
In particular, $\s_0(A)\subseteq\Cl(A)$.
\end{lem}

\begin{proof}
The forward implication is from~\cite[Proposition~3.4]{sather:divisor}.
For the converse, assume that $N$ is $A$-reflexive of rank 1 and $\Ext[A]{\geq 1}NN=0$.
Since $N$ is $A$-reflexive of rank~1, it represents an element of $\Cl(A)$, and in this group, we have
$[A]=[N]-[N]=[\Hom[A]NN]$.
Thus, $\Hom[A]NN\cong A$, and hence $N$ is semidualizing.
\end{proof}


We conclude this section with a version of $\s_0(R)$, for a noetherian ring $R$, that is better behaved when $R$ is non-local.  This is Definition~\ref{non-local}.

\begin{defn}
A finitely generated module $\fa$ over a noetherian ring $R$ is \emph{invertible}
if there is another finitely generated $R$-module $\fb$ such that $\Otimes\fa\fb\cong R$. 
Equivalently, $\fa$ is invertible if and only if it is finitely generated and projective of rank 1.
When these conditions are satisfied, the module $\fb=\Hom[R]\fa R$ satisfies $\Otimes\fa\fb\cong R$.
The \emph{Picard group} $\Pic(R)$ of a noetherian ring $R$ is the set of isomorphism classes
of invertible $R$-modules with operation $[\fa]+[\fb]=[\fa\otimes_R\fb]$.
\end{defn}

If $\vf\colon R\to S$ is a homomorphism of noetherian rings, then there is a well-defined group homomorphism $\Pic(\vf)\colon\Pic(R)\to\Pic(S)$ given by $[\fa]\mapsto[\Otimes \fa S]$.  (The argument is straightforward: for the well-definedness, show that if $\fb$ represents the inverse of $\fa$ in $\Pic(R)$, then 
$\Otimes \fb S$ represents the inverse of $\Otimes \fa S$ in $\Pic(S)$.) 

\begin{disc}\label{disc191101a} There is a natural inclusion $\Pic(R)\subseteq\s_0(R)$ for each noetherian ring $R$.
(This follows, e.g., from the fact that the semidualizing property is local and the Picard group of a local ring is trivial.
Indeed, if $R$ is not necessarily local and $[\fa]\in\Pic(R)$, then $\fa_\fm\cong R_{\m}$ for each maximal $\m$;
thus, $\fa_\m$ is semidualizing over $R_\m$ for all $\m$, so $[\fa]\in\s_0(R)$.)
%

Moreover, another localization argument shows that there is a well-defined action of $\Pic(R)$ on $\s_0(R)$ given by
$[\fa]+[C]=[\Otimes\fa C]$.
This action controls the lack of antisymmetry in $\s_0(R)$: for $[C],[M]\in\s_0(R)$, one has
$[M]\tri[C]\tri[M]$ if and only if $[C]=[\fa]+[M]$ for some $[\fa]\in\Pic(R)$.
See~\cite{frankild:rbsc} for details.  In the case that there is a flat homomorphism $\vf\colon R\to S$ of noetherian rings, then 
the map $\s_0(\vf)$ of Remark~\ref{disc:flat1} is equivariant. 
(Recall that if additive groups $G$ and $H$ act on sets $U$ and $V$, respectively, and $\psi\colon G\to H$ is 
a group homomorphism, then a function $\phi\colon U\to V$ is
\emph{equivariant} with respect to these actions if 
$\phi(g+u)=\psi(g)+\phi(u)$ for all $g\in G$ and all $u\in U$.)
\end{disc}

\begin{defn}\label{non-local}
For a noetherian ring $R$, let $\ol\s_0(R)$ denote the orbit space of $\s_0(R)$ under the action of $\Pic(R)$.
In other words,
$\ol\s_0(R)=\s_0(R)/\sim$ where $[M]\sim[C]$ provided that $[C]=[\fa]+[M]$ for some $[\fa]\in\Pic(R)$.
The orbit (i.e., equivalence class) of $[C]$ in $\ol\s_0(R)$ is denoted $\ec$.
For $\ec[M],\ec\in\sbar$ we write $\ec\tri\ec[M]$ whenever $[C]\tri[M]$ in $\s_0(R)$, i.e., provided that $M$ is totally $C$-reflexive.
\end{defn}

\begin{disc} \label{disc191101b}
Let $R$ be a noetherian ring.
\begin{enumerate}[(a)]
\item\label{disc191101b1} If $R$ is local, then $\ol\s_0(R)=\s_0(R)$ since $\Pic(R)$ is trivial.
\item\label{disc191101b2} If $R$ is semilocal, then $\ol\s_0(R)$ is finite
by~\cite[Theorem~4.6]{nasseh:gart}.
It is suspected that $\ol\s_0(R)$ is finite even when $R$ is not semilocal.
\item\label{disc191101b3} The relation $\tri$ on $\sbar$ is well-defined, reflexive, and antisymmetric by Remark~\ref{disc191101a}.
It is suspected that $\tri$ is also transitive on $\sbar$.
\item\label{disc191101b4} 
If $\vf\colon R\to S$ is a flat homomorphism of noetherian rings, then the fact that the map 
$\s_0(\vf)\colon \s_0(R)\to\s_0(S)$ is equivariant with respect to the Picard group actions implies that
the rule $\ec[C]\mapsto\ec[\Otimes CS]$ describes a well-defined function $\ol\s_0(\vf)\colon \ol\s_0(R)\to\ol\s_0(S)$.
\end{enumerate}
\end{disc}

\section{Localizations of Semidualizing Modules}
\label{sec191006b}

The work in~\cite{sather:divisor,SWSeSpP4} focuses on the case of graded normal domains
where the semidualizing modules are particularly well-behaved, as we see next.

\begin{fact}
\label{prop190312a}
Let $A=\oplus_{i=0}^\infty A_i$ be a  graded normal domain with $(A_0,\m_0)$  local (e.g., a field), and consider the maximal ideal 
$\m=\m_0\oplus A_+=\m_0\oplus A_1\oplus A_2\oplus\cdots$.
Then the natural map $\Cl(A)\to\Cl(A_\m)$ is an isomorphism and every element of $\Cl(A)$ is represented by a homogeneous
rank-1 reflexive module.
(Argue as in~\cite[Corollary~10.3]{fossum:dcgkd} using, e.g., \cite[Lemma~2.14]{frankild:rrhffd}.)
Furthermore, the natural map $\s_0(A)\to\s_0(A_\m)$ is a perfectly relation-respecting bijection, and every semidualizing $A$-module is graded
by~\cite[Proposition~3.12 and its proof]{sather:divisor}. 
In particular, it follows that $|\s_0(A)|=|\s_0(A_\m)|<\infty$ because of Remark~\ref{disc191006b}.
\end{fact}

Without the graded assumption, the  natural map $\s_0(A)\to\s_0(A_\m)$ from Fact~\ref{prop190312a} fails to be bijective, as is evinced by the next two examples. 

\begin{ex}[\protect{\cite[Remark~4.2.12]{sather:sdm}}]
\label{ex190329a}
Let $G$ be a non-zero abelian group. 
A result of Claborn~\cite{claborn} states that there is a Dedekind domain $D$ such that $\Pic(D)\cong G$. 
The fact that $D$ is a Dedekind domain implies that $\Cl(D)=\Pic(D)$, thus,
Lemma~\ref{lem190103a} and Remark~\ref{disc191101a} imply
$$\Cl(D)=\Pic(D)\subseteq\s_0(D)\subseteq\Cl(D).$$
As a result, $|\s_0(D)|=|G| > 1$. However, the Dedekind domain assumption further implies that for every maximal ideal $\m\subseteq D$,
the localization $D_\m$ is regular, so 
$|\s_0(D_\m)|=1$ by~\cite[(8.6)~Corollary]{christensen:scatac}. In particular, the map $\s_0(D)\to\s_0(D_\m)$ is not injective for any $\m$. 
\end{ex}

We are interested in understanding when the  natural map $\s_0(A)\to\s_0(A_\m)$ is bijective. 
Before posing an official question, we discuss some more limitations.

\begin{disc}\label{disc190329a}
Let $R$ be a  noetherian ring, and let $\m$ be a maximal ideal of $R$.
If the natural map $j\colon \s_0(R)\to\s_0(R_\m)$ is injective, then $\Pic(R)=0$ because 
$\Pic(R)\subseteq\s_0(R)$ implies $\Pic(R)\cong j(\Pic(R))\subseteq\Pic(R_\m)=0$. 
From this, it follows that $\sbar[R]=\s_0(R)$.
\end{disc}

From this remark, we see next that 
the  natural map  $\s_0(A)\to\s_0(A_\m)$ can fail to be bijective
for standard graded integral domains that are not normal.

\begin{ex}\label{ex190329b}
Let $\sfk$ be a field. Coykendall~\cite[Example~2.3]{MR1335714} shows that the graded integral domain $R=\sfk[s^2,s^3,w]\cong\sfk[x,y,w]/(x^3-y^2)$ satisfies
$\Pic(R)\neq 0$. This ring is not standard graded, but one can homogenize it to obtain the standard graded integral domain 
$$A=\sfk[s^2t,s^3t,t,w]\cong\sfk[x,y,z,w]/(x^3-y^2z).$$
As in~\cite[Example~2.3]{MR1335714}, one checks that 
the fractionary ideal $I=(1+stw,s^2tw^2)A$ represents a non-zero element of $\Pic(A)$, so
$\Pic(A)\neq 0$. 
Thus, for each maximal ideal $\m\subseteq A$, e.g., for $\m$ the irrelevant maximal ideal of $A$, the natural map $j\colon \s_0(A)\to\s_0(A_\m)$ is not injective.
\end{ex}

With the above information in mind, we pose the following.

\begin{question}\label{q190312a}
Let $A=\oplus_{i=0}^\infty A_i$ be a quotient of a standard graded polynomial ring in finitely many indeterminates over a field $A_0$ by a 
homogenous ideal $I$,
and set $\m=A_+=\oplus_{i=1}^\infty A_i$.
Under what conditions must the natural map $j\colon\s_0(A)\to\s_0(A_\m)$ be a bijection?
In particular, when must $\s_0(A)$ be finite?
If $I$ is a (square-free) monomial ideal, must the natural map $j\colon\s_0(A)\to\s_0(A_\m)$ be a bijection, and must $\s_0(A)$ be finite?
\end{question}

\begin{disc}\label{disc190329ax}
Let $A$ be a quotient of a standard graded polynomial ring in finitely many indeterminates over a field $A_0$ by a square-free monomial ideal. 
In light of Remark~\ref{disc190329a},
circumstantial evidence for injectivity in the last part of Question~\ref{q190312a} comes from~\cite[Corollary~4.14]{brenner:picard} which says that 
$\Pic(A)=0$. 

In this situation, one might try 
to answer Question~\ref{q190312a}, as follows. 
Let $C$ and $M$ be semidualizing $A$-modules such that $j([C])=j([M])$, i.e., such that $C_\m\cong M_\m$. 
It follows that
$$\Hom[A]{M}{C}_\m\cong\Hom[A_\m]{M_\m}{C_\m}\cong\Hom[A_\m]{C_\m}{C_\m}\cong A_\m.$$
It would be nice to conclude that $\Hom[A]{M}{C}\cong A$;
however, it is not clear how to conclude this. If one knew that $\Hom[A]{M}{C}$ were semidualizing over $A$ (or if one knew the stronger condition
of $M$ being totally $C$-reflexive), then there might be a chance;
however, it is not clear how to conclude this.

On the other hand, if $C$ and $M$ are graded over $A$, then~\cite[Proposition~2.11(5)]{SWSeSpP2} implies that $C\cong M$, as desired. Thus, we add
the next question.
\end{disc}

\begin{question}\label{q190329a}
Let $A$ be as in Remark~\ref{disc190329ax}.
Must every semidualizing $A$-module be graded?
\end{question}

We close this section with one more question, preceded by motivating discussion.

\begin{disc}\label{disc190813a}
Let $A$ be a normal domain, and let $U\subseteq A$ be a muliplicatively closed subset generated by prime elements. 
Nagata's theorem~\cite[Corollary~7.2]{fossum:dcgkd} tells us that the top horizontal map in the following commutative diagram is bijective.
$$\xymatrix{\Cl(A)\ar[r]^-\cong&\Cl(U^{-1}A)\\
\s_0(A)\ar[r]\ar@{^(->}[u]
&\s_0(U^{-1}A)\ar@{^(->}[u]
}$$
A straightforward diagram chase shows that the bottom horizontal map is injective. 
However, we do not know whether this map is bijective. 
(Note that one must have some assumptions on $A$ and $U$ in order for this to be so because 
of~\cite[Theorem~B]{nasseh:ahplrdmi}.)
Several of our proofs in~\cite{SWSeSpP4} would be simplified if this map were bijective.
Thus, we pose the following question.
\end{disc}

\begin{question}\label{q190103a}
Let $A$ be a normal domain, and let $U\subseteq A$ be a muliplicatively closed subset generated by prime elements. 
Must the natural map $\s_0(A)\to\s_0(U^{-1}A)$ be bijective?
\end{question}

\section{Tensor Products and Semidualizing Modules} \label{sec181225a}

The following assumption is somewhat restrictive, though it applies to several cases of interest, as we show below.

\begin{assumption}\label{ass181224a}
For the remainder of this paper, let $D$ be a 
principal ideal domain. 
For $i=1,2$ let $R_i$ be a noetherian integral domain equipped with a ring monomorphism
$f_i\colon D\to R_i$.
Set $T=R_1\otimes_DR_2$ equipped with the natural ring homomorphisms $g_i\colon R_i\to T$, 
which make the following diagram commute.
\begin{equation}\begin{split}\label{diag181225a}
\xymatrix{D\ar[r]^-{f_1}\ar[d]_{f_2}&R_1\ar[d]^-{g_1}\\
R_2\ar[r]_{g_2}&T
}\end{split}\end{equation}
\end{assumption}

The next example  shows  how we apply the results of this paper in~\cite{SWSeSpP4}.

\begin{ex}\label{ex181224c}
Let $X=(X_{ij})$ be an $m\times n$ matrix of indeterminates.
A \emph{ladder} in $X$ is a subset $Y$ satisfying the following property:
if $X_{ij},X_{pq}\in Y$ satisfy $i\leq p$ and $j\leq q$, then $X_{iq},X_{pj}\in Y$.  
As in~\cite[p.~121(b)]{Co}, to avoid trivialities, we assume without loss of generality that $X_{m1},X_{1n}\in Y$ and furthermore
that each row of $X$ contains an element of $Y$, as does  each column of $X$.
Given an integral domain $A$, let $I_t(Y)$ denote the ideal in the polynomial ring $A[Y]$
generated by the $t \times t$ minors of $X$ lying entirely in $Y$, and set $A_t(Y)=A[Y]/I_t(Y)$. 
The special case where $Y=X$ and $A$ is a normal domain is discussed in the introduction.

Let $D$ be a 
principal ideal domain contained in $A$ as a subring;
see Remark~\ref{disc181224a}.
Then the inclusions $D\subseteq A=R_1$ and $D\subseteq D_t(Y)=R_2$ satisfy Assumption~\ref{ass181224a}, and we have
$T=A\otimes_DD_t(Y)\cong A_t(Y)$.
\end{ex}

Here is another example, for perspective.

\begin{ex}\label{ex191006a}
Let $S$ be an additive monoid, and for any integral domain $A$ let $A[S]$ denote the associated semigroup ring;
this is the free $A$-module with basis $\{x^s\mid s\in S\}$ with multiplication induced by the rule $x^sx^t=x^{s+t}$ for all $s,t\in S$. 
If $D$ is as in Example~\ref{ex181224c}, then the inclusions 
$D\subseteq A=R_1$ and $D\subseteq D[S]=R_2$ satisfy Assumption~\ref{ass181224a}, and we have
$T=A\otimes_DD[S]\cong A[S]$.
\end{ex}

The following  is a well-known version of the K\"unneth formula. 

\begin{lem}\label{fact181224b}
Let $M_i, N_i$ be $R_i$-modules, for $i=1,2$, such that $M_i$ is finitely generated. Assume that one of the $M_i$ is torsion-free over $R_i$. 
Assume further that 
\begin{enumerate}[\rm(1)]
\item\label{fact181224b1} $D$ is a field, or 
\item\label{fact181224b2} for $j=1$ or $j=2$ the $R_j$-module $\Ext[R_j]{n}{M_j}{N_j}$ is torsion-free for all $n\geq 0$.
\end{enumerate}
Then for all $n\geq 0$ we have $T$-module isomorphisms
$$\Ext[T]{n}{M_1\otimes_DM_2}{N_1\otimes_DN_2}
\cong\bigoplus_{p+q=n}\Ext[R_1]{p}{M_1}{N_1}\otimes_D\Ext[R_2]{q}{M_2}{N_2}.$$
In particular, this implies that
$$\Hom[T]{M_1\otimes_DM_2}{N_1\otimes_DN_2}
\cong\Hom[R_1]{M_1}{N_1}\otimes_D\Hom[R_2]{M_2}{N_2}.$$
\end{lem}

The result below shows how our assumptions transform semidualizing $R_i$-modules into 
semidualizing $T$-modules. Example~\ref{ex181224c1} shows how this applies to the examples above.

\begin{prop}\label{prop181225a}
Assume that $T$ is noetherian,
and for $i=1,2$ let $C_i$ be a semidualizing $R_i$-module.
Then $C_1\otimes_DC_2$ is a semidualizing $T$-module.
Thus, there exists a well-defined, relation-respecting map $\s_0(g_1)\otimes\s_0(g_2)\colon\s_0(R_1)\times\s_0(R_2)\to\s_0(T)$ given by
$([C_1],[C_2])\mapsto [C_1\otimes_DC_2]$.
If $D$ is a field, then this map is injective. 
\end{prop}

\begin{proof}
Argue as in~\cite[Theorem~1.2]{altmann:sdmtp} using Lemma~\ref{fact181224b}, and invoke~\cite[Proposition~2.12]{SWSeSpP2}.
\end{proof}

\begin{ex} \label{ex181224c1}
With notation as in Examples~\ref{ex181224c} and~\ref{ex191006a},
Proposition~\ref{prop181225a} yields well-defined, relation-respecting functions
$\s_0(A)\times\s_0(D_t(Y))\to\s_0(A_t(Y))$
and
$\s_0(A)\times\s_0(D[S])\to\s_0(A[S])$, respectively, 
given by 
$([C_1],[C_2])\mapsto [C_1\otimes_D C_2]$.
\end{ex}

We conclude this section by observing that the map in Proposition~\ref{prop181225a} induces a map
$\ol\s_0(R_1)\times\ol\s_0(R_2)\to\ol\s_0(T)$. 

\begin{prop}\label{prop181225a'}
Assume that $T$ is noetherian.
\begin{enumerate}[\rm(a)]
\item\label{prop181225a'1} The map $\s_0(g_1)\otimes\s_0(g_2)\colon\s_0(R_1)\times\s_0(R_2)\to\s_0(T)$
from Proposition~\ref{prop181225a}
induces a  homomorphism $\Pic(g_1)\otimes\Pic(g_2)\colon\Pic(R_1)\times\Pic(R_2)\to\Pic(T)$ with
$$([\fa_1],[\fa_2])\mapsto [\fa_1\otimes_D\fa_2]= [\fa_1\otimes_{R_1}T]+[T\otimes_{R_2}\fa_2]=[\fa_1\otimes_DR_2]+[R_1\otimes_D\fa_2].$$
\item\label{prop181225a'2} The map $\s_0(g_1)\otimes\s_0(g_2)$ is equivariant for the action of
$\Pic(R_1)\times\Pic(R_2)$ on $\s_0(R_1)\times\s_0(R_2)$
and the action of $\Pic(T)$ on $\s_0(T)$.
\item\label{prop181225a'3} The map $\s_0(g_1)\otimes\s_0(g_2)$ also induces a well-defined, relation-respecting map 
$\ol\s_0(g_1)\otimes\ol\s_0(g_2)\colon\ol\s_0(R_1)\times\ol\s_0(R_2)\to\ol\s_0(T)$ with
$(\ec[C_1],\ec[C_2])\mapsto \ec[C_1\otimes_DC_2]$.
\end{enumerate}
\end{prop}

\begin{proof}
Let $\phi = \s_0(g_1)\otimes\s_0(g_2)$ and $\psi=\Pic(g_1)\otimes\Pic(g_2)$.
Note that for invertible $R_i$-modules $\fa_i$, the $T$-module $\Otimes[D]{\fa_1}{\fa_2}$ is invertible.
Indeed, if $[\fb_i]=-[\fa_i]$ in $\Pic(R_i)$ for $i=1,2$, then
$\Otimes[R_i]{\fb_i}{\fa_i}\cong R_i$ for $i=1,2$, and so
Fact~\ref{fact181224a} can be used to show that, as desired,
\begin{align*}
\Otimes[T]{(\Otimes[D]{\fb_1}{\fb_2})}{(\Otimes[D]{\fa_1}{\fa_2})}
&\cong\Otimes[D]{(\Otimes[R_1]{\fb_1}{\fa_1})}{(\Otimes[R_2]{\fb_2}{\fa_2})}
\cong\Otimes[D]{R_1}{R_2}\cong T.
\end{align*}
Thus, $\psi$ is well-defined.
Similar arguments show that $\psi$ is a group homomorphism,  the displayed equalities in part~\eqref{prop181225a'1} hold,
and  $\phi$ is equivariant. 
This establishes parts~\eqref{prop181225a'1} and~\eqref{prop181225a'2} of the result, and
part~\eqref{prop181225a'3}
follows directly.
\end{proof}

\section{Divisor Class Groups and Semidualizing Modules}
\label{sec191006c}

\begin{assumption}\label{ass181225a}
Continue with Assumption~\ref{ass181224a}.
For the rest of this paper, assume in addition that the rings $R_1$, $R_2$, and $T$ are  normal domains.
\end{assumption}

The point of this section is to describe how divisor classes and semidualizing modules over the rings $R_i$
contribute divisor classes and semidualizing modules over the tensor product $T$. In~\cite{SWSeSpP4},
we specialize to the situation of Example~\ref{ex181224c} and show that these are the only 
divisor classes and semidualizing modules over $T=A_t(Y)$.

\begin{lem}\label{lem181225a}
The rule
$(\Cl(g_1)+\Cl(g_2))([\fa_1],[\fa_2])= [\fa_1\otimes_{R_1}T]+[T\otimes_{R_2}\fa_2]$
describes
a well-defined homomorphism $\Cl(g_1)+\Cl(g_2)\colon\Cl(R_1)\times\Cl(R_2)\to\Cl(T)$
such that $(\Cl(g_1)+\Cl(g_2))([\fa_1],[\fa_2])=[\fa_1\otimes_DR_2]+[R_1\otimes_D\fa_2]$.
\end{lem}

\begin{proof}
Since each map $g_i\colon R_i\to T$ is flat, there are well-defined group homomorphisms
$\Cl(g_i)\colon \Cl(R_i)\to\Cl(T)$ given by $[\fa_i]\mapsto[\fa_i\otimes_{R_i}T]$.
Sum these maps to get the desired homomorphism $\Cl(R_1)\times\Cl(R_2)\to\Cl(T)$.
The final equality in
the statement follows from tensor cancelation as in the first display of Fact~\ref{fact181224a}.
\end{proof}

\begin{lem}\label{lem181224a}
Let $[\fa_i]\in\Cl(R_i)$ for $i=1,2$.
If $D$ is a field, or $[\fa_i]\in\s_0(R_i)$ for $i=1,2$,
then $\fa_1\otimes_D \fa_2$ is reflexive over $T$ and
$(\Cl(g_1)+\Cl(g_2))([\fa_1],[\fa_2])= [\fa_1\otimes_D\fa_2]$.
\end{lem}

\begin{proof}
We first prove that $\fa_1\otimes_D \fa_2$ is reflexive over $T$.

Assume in this paragraph that $D$ is a field. 
For $i=1,2$, the modules $\fa_i$ and $\Hom[R_i]{\fa_i}{R_i}$ are reflexive over $R_i$, hence torsion-free.
Thus, since $T=R_1\otimes_DR_2$, Lemma~\ref{fact181224b} explains the first two steps in the following display.
\begin{align*}
\Hom[T]{\Hom[T]{\fa_1\otimes_D \fa_2}{T}}{T}\hspace{-1in}\\
&\cong
\Hom[T]{\Hom[R_1]{\fa_1}{R_1}\otimes_D\Hom[R_2]{\fa_2}{R_2}}{T}\\
&\cong
\Hom[R_1]{\Hom[R_1]{\fa_1}{R_1}}{R_1}\otimes_D\Hom[R_2]{\Hom[R_2]{\fa_2}{R_2}}{R_2}\\
&\cong
\fa_1\otimes_D\fa_2
\end{align*}
The last step is from the reflexive assumption for each $\fa_i$.
This establishes the reflexivity of $\fa_1\otimes_D \fa_2$ in the first case. 

Next, assume that $[\fa_i]\in\s_0(R_i)$ for $i=1,2$.
Then Proposition~\ref{prop181225a} implies that 
$\fa_1\otimes_D \fa_2$ is semidualizing over $T$.
Thus, $\fa_1\otimes_D \fa_2$ is reflexive over $T$ by Lemma~\ref{lem190103a}
This establishes the reflexivity of $\fa_1\otimes_D \fa_2$ in the second case.

To complete the proof, let $([\fa_1],[\fa_2]) \in \s_0(R_1) \times \s_0(R_2)$.  First using the definition of $\Cl(g_1)+\Cl(g_2)$ from Lemma~\ref{lem181225a} and then the definition of addition in $\Cl(T)$, we have:
\begin{align*}
(\Cl(g_1)+\Cl(g_2))([\fa_1],[\fa_2])
&= [\fa_1\otimes_{R_1}T]+[T\otimes_{R_2}\fa_2]\\
&= [\Hom[T]{\Hom[T]{(\fa_1\otimes_{R_1}T)\otimes_T(T\otimes_{R_2}\fa_2)}{T}}{T}]\\
&= [\Hom[T]{\Hom[T]{\fa_1\otimes_D\fa_2}{T}}{T}]\\
&=[\fa_1\otimes_D\fa_2]
\end{align*}
Note that the third step is from Fact~\ref{fact181224a} and the last step is by reflexivity.
\end{proof}

In~\cite{SWSeSpP4}, in the context of Example~\ref{ex181224c}, we use certain localizations $S_i$ of $T$ 
that guarantee that the top horizontal map in~\eqref{diag181226z} is an isomorphism;
then we show that they also make the maps in the middle triagram into bijections.

\begin{thm}\label{lem181226a}
For $i=1,2$ let $h_i\colon T\to S_i$ be a flat ring homomorphism with $S_i$ a normal domain.
Consider the following  diagram where the vertical maps are the natural inclusions and projections.
\begin{equation}
\label{diag181226z}
\begin{split}
\xymatrix@C=4em{
\Cl(R_1)\times\Cl(R_2)
\ar[rr]^-{\Cl(h_1\circ g_1)\times\Cl(h_2\circ g_2)}
\ar[rd]_>>>>>>>{\Cl(g_1)+\Cl(g_2)\ \ }
&&
\Cl(S_1)\times\Cl(S_2)\\
&\Cl(T)
\ar[ru]_<<<<<<<{\ \ (\Cl(h_1),\Cl(h_2))}
\\
\s_0(R_1)\times\s_0(R_2)\ar@{^(->}[uu]\ar@{->>}[dd]
\ar '[r][rr]^<{\s_0(h_1\circ g_1)\times\s_0(h_2\circ g_2)}
\ar[rd]_>>>>>>>{\s_0(g_1)\otimes\s_0(g_2)\ \ }
&&
\s_0(S_1)\times\s_0(S_2)\ar@{^(->}[uu]\ar@{->>}[dd]\\
&\s_0(T)\ar@{^(->}[uu]\ar@{->>}[dd]
\ar[ru]_<<<<<<<{\ \ \ (\s_0(h_1),\s_0(h_2))}
\\
\ol\s_0(R_1)\times\ol\s_0(R_2)
\ar '[r][rr]^<{\ol\s_0(h_1\circ g_1)\times\ol\s_0(h_2\circ g_2)}
\ar[rd]_-{\ol\s_0(g_1)\otimes\ol\s_0(g_2)\ \ }
&&
\ol\s_0(S_1)\times\ol\s_0(S_2)\\
&\ol\s_0(T)
\ar[ru]_-{\ \ \ (\ol\s_0(h_1),\ol\s_0(h_2))}
}
\end{split}
\end{equation}
Then the six quadrilateral faces of diagram~\eqref{diag181226z} commute.
\end{thm}

\begin{proof}
The well-definedness of the map $\s_0(g_1)\otimes\s_0(g_2)$ is given in Proposition~\ref{prop181225a}.
Note that the top rectangular (back) face of diagram~\eqref{diag181226z} is just the product of the following diagrams, each of which commutes
by Remarks~\ref{disc:flat1} and \ref{disc:flat2}.
$$\xymatrix@C=4em{
\Cl(R_1)
\ar[r]^-{\Cl(h_1\circ g_1)}
&\Cl(S_1)
&\Cl(R_2)
\ar[r]^-{\Cl(h_2\circ g_2)}
&\Cl(S_2)
\\
\s_0(R_1)
\ar@{^(->}[u]
\ar[r]^-{\s_0(h_1\circ g_1)}
&\s_0(S_1)
\ar@{^(->}[u]
&\s_0(R_2)
\ar@{^(->}[u]
\ar[r]^-{\s_0(h_2\circ g_2)}
&\s_0(S_2)
\ar@{^(->}[u]
}$$
Thus, the top rectangular (back) face of diagram~\eqref{diag181226z} commutes.
The top right-hand parallelogram commutes similarly, 
and the top left-hand parallelogram commutes by Lemma~\ref{lem181224a}.
The bottom quadrilateral faces commute by definition.
\end{proof}

\begin{disc}
In general, there is no reason to expect the triagrams in~\eqref{diag181226z} to commute. 
However, under the very special circumstances of~\cite{SWSeSpP4}, they do commute.
\end{disc}

\providecommand{\bysame}{\leavevmode\hbox to3em{\hrulefill}\thinspace}
\providecommand{\MR}{\relax\ifhmode\unskip\space\fi MR }
\providecommand{\MRhref}[2]{%
  \href{http://www.ams.org/mathscinet-getitem?mr=#1}{#2}
}
\providecommand{\href}[2]{#2}

\end{document}